\documentclass[12pt]{amsart}
\usepackage{geometry}   
\usepackage[colorlinks,citecolor = red, linkcolor=blue,hyperindex]{hyperref}
\usepackage{euscript,verbatim, mathrsfs}
\usepackage[psamsfonts]{amssymb}
\usepackage{bbm}
\usepackage{graphicx}
\usepackage{float}
\usepackage{float, tikz}

 \usepackage{extarrows}
 \usepackage[all, cmtip]{xy}

\usepackage{upref, xcolor, dsfont}
\usepackage{amsfonts,amsmath,amstext,amsbsy, amsopn,amsthm}
\usepackage{enumerate}

\usepackage{url}

\usepackage{mathtools}
\usepackage{bookmark}

 \usepackage{euscript}
\usepackage{helvet}         
\usepackage{courier}        
\usepackage{type1cm}        
\usepackage{multicol}        
\usepackage[bottom]{footmisc}

\newtheorem{theorem}{Theorem}[section]
\newtheorem*{theorem*}{Theorem A}
\newtheorem{lemma}[theorem]{Lemma}

\newtheorem*{definition*}{Definition}
\newtheorem*{remark*}{Remark}

\newtheorem*{observation*}{Observation}

\newtheorem*{assumption*}{Assumption}

\theoremstyle{definition}

\theoremstyle{remark}
\newtheorem{remark}{Remark}[section]

\newcommand{\Z}{\mathbb{Z}}

\newcommand{\D}{\mathbb{D}}

\newcommand{\T}{\mathbb{T}}

\newcommand{\K}{\mathcal{K}}

\begin{document}

\title{Weighted inequality of integral operators induced by Hardy kernels}

\author
{Zipeng Wang}
\address{Zipeng Wang\\
	College of Mathematics and Statistics, Chongqing University\\
	Chongqing, 401331, China}
\email{zipengwang2012@gmail.com}

\author
{Kenan Zhang}
\address
{Kenan Zhang\\
	School of Mathematical Sciences, Fudan University\\
	Shanghai, 200433, China}
\email{knzhang21@m.fudan.edu.cn}

\thanks{}

\begin{abstract}
For doubling weights, we obtain a necessary and sufficient condition such that the one weighted inequality of the integral operator induced by Hardy kernels on the unit disk holds. This confirms a conjecture by Guo and Wang in such situations. 
\end{abstract}


\maketitle

\setcounter{equation}{0}

\section{Introduction}
Let $\alpha>0$ and $\mathbb{D}$ be the unit disk in the complex plane. Cheng-Fang-Wang-Yu \cite{CFWY} considered the following integral operators on the unit disk
\[
	(K_\alpha f)(z)=\int_{\D}f(\lambda)\K_\alpha(z,\lambda)dA(\lambda),
\]
where
\[
\K_\alpha(z,\lambda)=\frac{1}{(1-z\overline{\lambda})^\alpha}.
\]

\vspace{0.1in}

The intergal kernel $\K_\alpha$ is a typical reproducking kernel, in particular unitary invariance kernels \cite{GHX}, for certain Hilbert spaces of holomorphic functions on the unit disk. For example, the operator $\K_2$ is the Bergman projection and $\K_1$ is an operator induced by the Hardy kernel. Cheng-Fang-Wang-Yu in \cite{CFWY} descripted the range of $(p,q)\in [1,+\infty]\times[1,+\infty]$ such that $K_\alpha$ is bounded from $L^p(\mathbb{D})$ to $L^q(\mathbb{D})$.

\vspace{0.1in}
A weight $\omega$ is a positive integrable function on the unit disk. For $p\geq 1$, the weighted space $L^p(\omega)$ is the Banach space
\[
\{f:f\mbox{ is a complex-valued function on $\D$ with }\|f\|_{L^p(\omega)}<\infty\},
\]
where
\[
\|f\|_{L^p(\omega)}:=
\bigg(\int_\mathbb{D}|f(z)|^p\omega(z)dA(z)\bigg)^{1/p}
\]
and $dA(z)=\frac{1}{\pi}dxdy$ is the normalized area measure. 

\vspace{0.1in}

The one weight inequality for $K_\alpha$ is to characterize the weight $\omega$ such that $$K_\alpha:L^p(\omega)\to L^p(\omega)$$ is bounded. When $\alpha=2$, the corresponding one weight inequality is the one weight problem of Bergman projection, and is well understood by the remarkable Bekoll\'e-Bonami weight \cite{FW2015}.

To the best of our knowledge, except the Bergman projection, all the other one weight inequality problems of the integral operator $K_\alpha$ remain open.

Let $I\subset\mathbb{T}$ be an interval on the unit circle. The associated Carleson box $Q_I$ in the unit disk is
\[
Q_I=\{z\in\mathbb{D}:z/|z|\in I,1-|I|<|z|<1\},
\]
where $|I|$ is the normalized arc length of the interval $I$.  

Let  $|Q_I|=\int_{Q_I}dA$ and $|Q_I|_\omega=\int_{Q_I}\omega(z) dA(z)$.
In 2015, Guo-Wang \cite{G-W} proved that if a weight $\omega$ on the unit disk satisfies
\[
\sup_{I:I\subset\mathbb{T}}\frac{|Q_I|_\omega|Q_I|_{\omega^{-1}}}{|Q_I|^{3/2}}<\infty,
\]
then $K_1$ is bounded on $L^2(\omega)$. They also conjectured that the necessary and sufficient condition for the boundedness of $K_1$ on $L^2(\omega)$ is
\[
\sup_{I:I\subset\mathbb{T}}\frac{|Q_I|_\omega|Q_I|_{\omega^{-1}}}{|Q_I|}<\infty.
\]

\vspace{0.1in}

In this note, under the assumption on the doubling condition, we confirm Guo-Wang's conjecture. Moreover, for any $\alpha>1$ we characterize the one weight inequality for $K_\alpha$ on the weighted space $L^p(\omega)$ for  $1<p<\infty$. 

Recall a measure $\mu$ on the unit disk satisfies the doubling conditions if there extists a constant $C>0$ such that $\mu(B(z,2r))\leq C\mu(B(z,r))$ for any $z\in\mathbb{D},r>0$, where $B(z,r)$ is the Euclidean disk on the complex plane with center $z$ and radius $r$.

\begin{theorem}\label{main}
	Let $\alpha>0$ and $1<p,p'<\infty$ with $\frac{1}{p}+\frac{1}{p'}=1$.
	Let $\omega$ be a weight on the unit disk $\D$ and $\sigma=\omega^{1-p'}$ be its dual weight. 
	Suppose that both $\omega$ and $\sigma$ are doubling weights. Then the operator $K_\alpha$ is bounded on $L^p(\omega)$ if and only if 
	\[
		[\omega]_{p,\alpha}:=\sup\limits_{I\subset \T,\ interval}\frac{|Q_I|_\omega |Q_I|_\sigma^{p-1}}{|Q_I|^{\frac{p \alpha}{2}}}<\infty.
	\]
\end{theorem}
\begin{remark}
The only if part holds without requiring the doubling conditions.  The if part remains true under the reverse doubling conditions \eqref{reverse-dou}, which are slightly weaker than the standard doubling condition.
\end{remark}

\section{The proof of Theorem \ref{main}}
\vspace{0.1cm}
\subsection{The necessity}
Assume that $K_\alpha$ is bounded on $L^p(\omega)$. 
By Sawyer's duality trick \cite{Sawyer1982}, the operator 
\[
(K_\alpha f)(z)=\int_{\D}f(\lambda)\K_\alpha(z,\lambda) dA(\lambda)
\]
is bounded on $L^p(\omega)$ if and only if
\[
(K_{\alpha,\sigma} f)(z):=(K_\alpha (\sigma f))(z)=\int_{\D}f(\lambda)\K_\alpha(z,\lambda)\sigma(\lambda) dA(\lambda)
\]
is bounded from $L^p(\sigma)$ to $L^p(\omega)$. Let
\[
A=\|K_{\alpha,\sigma}\|_{L^p(\sigma)\to L^p(\omega)}.
\]
Then the following weak-type inequality holds
\begin{equation}\label{equ}
	|\{z\in\D:|(K_{\alpha,\sigma} f)(z)|>\gamma\}|_\omega\leqslant\frac{A^p}{\gamma^p}\int_{\D}|f(z)|^p\sigma(z) dA(z).
\end{equation}

\begin{lemma}\label{necessity}
	For $\alpha>0$, there exist intervals $I,J\in\T$ such that $|J|=|I|$, and for each $g\geqslant0$ supported in $Q_I$ and each $z\in Q_J$, we have
	\[
	|(K_\alpha g)(z)|\geqslant \frac{C_1}{|Q_I|^{\frac{\alpha}{2}}}\int_{\D}g(\lambda) dA(\lambda),
	\]
	where $C_1$ depentent only on $\alpha$.
\end{lemma}
\begin{proof}
	Denote the center of $Q_I$ by $c$. 
	Then for fixed $g\geqslant0$ supporting on $Q_I$, we have
	\begin{equation*}
		\begin{aligned}
			(K_\alpha g)(z)&=\int_{\D}g(\lambda)\K_\alpha(z,\lambda)\sigma(\lambda) dA(\lambda)\\
			&=\frac{1}{(1-z\overline{c})^\alpha}\int_{\D}g(\lambda) dA(\lambda)+\int_{\D}\left(\frac{1}{(1-z\overline{\lambda})^\alpha}-\frac{1}{(1-z\overline{c})^\alpha}\right)g(\lambda) dA(\lambda).
		\end{aligned}
	\end{equation*}
	Consider the bracketed part in the latter term, by the mean value theorem
	\begin{equation*}
		\begin{aligned}
			\left|\frac{1}{(1-z\overline{\lambda})^\alpha}-\frac{1}{(1-z\overline{c})^\alpha}\right|
			&=\left|\frac{1}{(1-z\overline{c})^\alpha}\right|\left|\frac{(1-z\overline{c})^\alpha-(1-z\overline{\lambda})^\alpha}{(1-z\overline{\lambda})^\alpha}\right|\\
			&=\frac{1}{|1-z\overline{c}|^\alpha}\frac{\alpha|z\overline{c}-z\overline{\lambda}||1-z\overline{\xi}|^{\alpha-1}}{|1-z\overline{\lambda}|^\alpha},
		\end{aligned}
	\end{equation*}
	where $\xi$ is a point on the segment connecting $c$ and $\lambda$.
	
	Here we estimate the second term in the last equation. We identify the circle $\T$ with the interval $[0,2\pi)$.
	Let $I=(0,\theta)$ and $J=(d\theta,(d+1)\theta)$ for some $d>2$ and $0<(d+1)\theta<\pi/2$.
	Assume $z=r_1 e^{i\xi_1}\in Q_J$ and $\lambda=r_2 e^{i\xi_2}\in Q_I$, then
	\begin{equation}\label{box}
		\begin{aligned}
			|1-z\overline{\lambda}|^2&=|1-r_1r_2e^{i(\xi_1-\xi_2)}|^2\\
			&=1+r_1^2r_2^2-2r_1r_2\cos(\xi_1-\xi_2)\\
			&=2r_1r_2(1-\cos(\xi_1-\xi_2))+(1-r_1r_2)^2\\
			&=4r_1r_2\sin^2\left(\frac{\xi_1-\xi_2}{2}\right)+(1-r_1r_2)^2\\
			&\geqslant \frac{4}{\pi}(1-\theta)^2(d-1)^2\theta^2
		\end{aligned}
	\end{equation}
	Since $d>2$ and $0<(d+1)\theta<\pi/2$, we have 
	\[
		1-\theta\geqslant1-\frac{\pi}{2(d+1)}\geqslant1-\frac{\pi}{6}.
	\]
	Let the constant $a=\frac{2}{\sqrt{\pi}}\left(1-\frac{\pi}{6}\right)$, we have
	\[
		|1-z\overline{\lambda}|\geqslant a(d-1)\theta.
	\]
	Besides, 
	\begin{equation*}
		\begin{aligned}
			|c-\lambda|^2
			&=|c|^2+r_2^2-2|c|r_2\cos\left(\frac{\theta}{2}-\xi_2\right)\\
			&=2|c|r_2\left(1-\cos\left(\frac{\theta}{2}-\xi_2\right)\right)+(|c|-r_2)^2\\
			&=4|c|r_2\sin^2\left(\frac{\theta}{2}-\xi_2\right)+(|c|-r_2)^2\\
			&\leqslant\theta^2+\frac{\theta^2}{4}\\
			&=\frac{5}{4}\theta^2,
		\end{aligned}
	\end{equation*}
	Then
	\[
		|c-\lambda|\leqslant \frac{\sqrt{5}}{2}\theta.
	\]
	Since
	\begin{equation*}
		\begin{aligned}
			|1-z\overline{\xi}|
			&\leqslant|1-z\overline{\lambda}|+|z| |\overline{\xi}-\overline{\lambda}|\\
			&\leqslant|1-z\overline{\lambda}|+|c-\lambda|,
		\end{aligned}
	\end{equation*}
	 we have
	\begin{equation*}
		\begin{aligned}
			\frac{|1-z\overline{\xi}|}{|1-z\overline{\lambda}|}
			&\leqslant 1+\frac{|c-\lambda|}{|1-z\overline{\lambda}|}\\
			&\leqslant 1+\frac{\sqrt{5}}{2a(d-1)}\\
			&\leqslant 1+\frac{\sqrt{5}}{2a}.
		\end{aligned}
	\end{equation*}
	Denote the constan $b=1+\frac{\sqrt{5}}{2a}$. Then we can choose $d$ large enough such that 
	\[
		\frac{\alpha|z\overline{c}-z\overline{\lambda}||1-z\overline{\xi}|^{\alpha-1}}{|1-z\overline{\lambda}|^\alpha}
		\leqslant\frac{\sqrt{5} b^{\alpha-1}\alpha}{2a(d-1)}\leqslant \frac{1}{2}.
	\]
	Hence,
	\[
		\left|\frac{1}{(1-z\overline{\lambda})^\alpha}-\frac{1}{(1-z\overline{c})^\alpha}\right|\leqslant\frac{1}{2}\frac{1}{|1-z\overline{c}|^\alpha}.
	\]
	Then by \eqref{box}, there exists a constant $C_1>0$ depentent only on $\alpha$ such that 
	\begin{equation*}
		\begin{aligned}
			|(K_\alpha g)(z)|&\geqslant\frac{1}{2}\frac{1}{|1-z\overline{c}|^\alpha}\int_{\D}g(\lambda) dA(\lambda)\\
			&\geqslant \frac{C_1}{|Q_I|^{\frac{\alpha}{2}}}\int_{\D}g(\lambda) dA(\lambda).
		\end{aligned}
	\end{equation*}
\end{proof}
Let $Q_I$ be a Carleson box on $\D$ small enough and $g(\lambda)=\sigma(\lambda)\chi_{Q_I}(\lambda)$.
Then by the weak-type estimate \eqref{necessity} there exists a Carleson box $Q_J$ such that for all $z\in Q_J$
\[
|(K_\alpha g)(z)|\geqslant \frac{C_1}{|Q_I|^{\frac{\alpha}{2}}}\int_{\D}g(\lambda) dA(\lambda)=C_1 \frac{|Q_I|_\sigma}{|Q_I|^{\frac{\alpha}{2}}}.
\]
Let $f(z)=\chi_{Q_J}(z)$, then $K_{\alpha,\sigma}f=K_\alpha g$ and by \eqref{equ}, we have
\[
|Q_J|_\omega
\leqslant \left|\left\{z\in\D:|(K_{\alpha,\sigma} f)(z)|>C_1 \frac{|Q_I|_\sigma}{|Q_I|^{\frac{\alpha}{2}}}\right\}\right|_\omega
\leqslant \left(\frac{A}{C_1}\right)^p\frac{|Q_I|^{\frac{p\alpha}{2}}}{|Q_I|_\sigma^p}\int_{Q_I}\sigma(z) dA(z).
\]
Denote $C_2=\left(\frac{A}{C_1}\right)^p$, then we have
\[
\frac{|Q_J|_\omega |Q_I|_\sigma^{p-1}}{|Q_I|^{\frac{p \alpha}{2}}}\leqslant C_2.
\]
Since $|I|=|J|$ by Lemma \ref{necessity}, 
\[
\sup\limits_{I\subset \T,\ interval}\frac{|Q_I|_\omega |Q_I|_\sigma^{p-1}}{|Q_I|^{\frac{p \alpha}{2}}}\leqslant C_2<\infty.
\]
This completes the proof of the necessary part.

\subsection{The Sufficienty}
We consider the following dyadic systems on the unit circle $\T$: 
\[
\mathcal{D}^0:=\left\{\left[\frac{2\pi k}{2^j},\frac{2\pi (k+1)}{2^j}\right):j\in\Z_+,k\in\Z_+,0\leqslant k\leqslant 2^j-1\right\},
\]
and
\[
\mathcal{D}^{\frac{1}{3}}:=\left\{\left[\frac{2\pi k}{2^j}+\frac{2\pi}{3},\frac{2\pi (k+1)}{2^j}+\frac{2\pi}{3}\right):j\in\Z_+,k\in\Z_+,0\leqslant k\leqslant 2^j-1\right\}.
\]
Denote the corresponding Carleson box system on the unit disk by
\[
\mathcal{Q}^\beta:=\{Q_I:I\in\mathcal{D}^\beta\},\quad \beta\in\{0,\frac{1}{3}\}.
\]
For each $\beta\in\{0,\frac{1}{3}\}$, we define an operator $T_\alpha^\beta$ on $L^1(\mathbb{D})$ by 
\[
(T_\alpha^\beta f)(z)=\sum\limits_{I\in\mathcal{D}^\beta}\frac{1}{|Q_I|^{\frac{\alpha}{2}}}\int_{Q_I}f(\lambda)dA(\lambda)\chi_{Q_I}(z),
\]
where $\chi_{Q_I}(z)$ is the characteristic function of $Q_I$. 

The following lemma illustrates that $K_\alpha$ can be pointwise controlled by $T_\alpha^\beta,\beta\in\{0,1/3\}$.
\begin{lemma}\label{lemma2.4}
	There exists a constant $C_3>0$ such that for any $f\geqslant0$ and $z\in\D$, we have
	\[
	|K_\alpha f(z)|\leqslant C_3\left((T_\alpha^0 f)(z)+(T_\alpha^{\frac{1}{3}} f)(z)\right).
	\]
\end{lemma}
\begin{proof}
	Fixed $\beta\in\{0,\frac{1}{3}\}$.
	Let $K_\alpha^\beta$ be the integral kernel of $T_\alpha^\beta$, i.e. for $z,\lambda\in\D$
	\[
	K_\alpha^\beta(z,\lambda)=\sum\limits_{I\in\mathcal{Q}^\beta}\frac{\chi_{Q_I}(z)\chi_{Q_I}(\lambda)}{|Q_I|^{\frac{\alpha}{2}}}.
	\]
	Hence it suffices to prove that for any $z,\lambda\in\D$, there exists a constant $C_3>0$ only depending on $\alpha$ such that
	\[
	|K_\alpha(z,\lambda)|\leqslant C_3\left(K_\alpha^0(z,\lambda)+K_\alpha^{\frac{1}{3}}(z,\lambda)\right).
	\]
	
	Recall the known relation between Bergman-type kernels and Carleson boxes, there exists a constant $C_4$ such that for any $z,\lambda\in\D$, there exists a Carleson box $Q_I$ such that $z,\lambda\in Q_I$ and
	\[
	\frac{1}{C_4}|Q_I|^{\frac{1}{2}}\leqslant|1-z\overline{\lambda}|\leqslant C_4|Q_I|^{\frac{1}{2}}.
	\]
	Then we have
	\[
	|K_\alpha(z,\lambda)|=\bigg|\frac{1}{(1-z\overline{\lambda})^\alpha}\bigg|\leqslant \frac{C_4^\alpha}{|Q_I|^{\frac{\alpha}{2}}}.
	\]
	Since 
	\[
	|Q_I|=
	\left\{\begin{aligned}
		&|I|^2-\frac{|I|^3}{2},\quad&&0<|I|\leqslant 1\\
		&\frac{|I|}{2},\quad&&1<|I|\leqslant 2\pi
	\end{aligned}\right.,
	\]
	we have $$\frac{1}{4\pi}|I|^2\leqslant|Q_I|\leqslant|I|^2.$$
	
	Applying Mei's elegant result \cite[Proposition 2.1]{Mei2003}, we see that for any interval $I\in\T$ there exists an interval $J\in\mathcal{D}^0\cup\mathcal{D}^{\frac{1}{3}}$ such that $I\in J$ and $|J|\leqslant 6|I|$.
	Hence, for any $z,\lambda\in\D$, 
	\begin{equation*}
		\begin{aligned}
			|K_\alpha(z,\lambda)|&\leqslant\frac{C_4^\alpha}{|Q_I|^{\frac{\alpha}{2}}}\\
			&\leqslant\frac{(2\sqrt{\pi}C_4)^\alpha}{|I|^{\alpha}}\\
			&\leqslant\frac{(12\sqrt{\pi}C_4)^\alpha}{|J|^{\alpha}}\\
			&\leqslant\frac{(12\sqrt{\pi}C_4)^\alpha}{|Q_J|^{\frac{\alpha}{2}}}\\
			&\leqslant C_3\left(K_\alpha^0(z,\lambda)+K_\alpha^{\frac{1}{3}}(z,\lambda)\right).
		\end{aligned}
	\end{equation*}
Here $C_3=(12\sqrt{\pi}C_4)^\alpha$.
This completes the proof.
\end{proof}

For a weight $\omega$ on $\D$, we also consider the maximal operator
\[
(M_\omega^\beta f)(z)=\sup\limits_{I\in\mathcal{D}^\beta}\frac{\chi_{Q_I}(z)}{|Q_I|_\omega}\int_{Q_I}|f(\lambda)|\omega(\lambda) dA(\lambda).
\]
By \cite[Theorem 5.7]{Tolsa}, we have
\begin{lemma}\label{max-norm}
	Let $1<p<\infty$ and $\omega$ be a weight on $\mathbb{D}$.
	Then there exists a constant $C_5$ independent of the weight $\omega$ and the dyadic systen $\mathcal{D}^\beta$ such that
	\[
	||M_\omega^\beta f||_{L^p(\omega)}\leqslant C_5||f||_{L^p(\omega)}.
	\]
\end{lemma}

By \cite[Theorem 5.8]{Tolsa}, \cite[Lemma 10]{FW2015} and Lemma \ref{max-norm}, we have the following Carleson embedding results.
\begin{lemma}\label{theorem2.5}
	Let $\omega$ be a weight on $\D$ with the reverse doubling property and let $\mathcal{Q}^\beta,\beta\in\{0,\frac{1}{3}\}$ be a Carleson system over $\D$. 
	Then there exists a constant $C_6>0$ such that for all $g\in L^p(\omega)$, we have
	\[
	\sum\limits_{I\in\mathcal{Q}^\beta}|Q_I|_\omega\left(\frac{1}{|Q_I|_\omega}\int_{Q_I}g(\lambda) \omega(\lambda) dA(\lambda)\right)^p
	\leqslant C_6 \int_{\D}|g(\lambda)|^p \omega(\lambda) dA(\lambda).
	\]
\end{lemma}

\vspace{0.1in}

Here we recall the definition of the reverse doubling property.
Given a Carleson box $Q_I$, set
$
	B_I:=\{z\in\D: \frac{z}{|z|}\in I, 1-\frac{|I|}{2}<|z|<1\},
$
a weight $\omega$ on the unit disk $\D$ has the reverse doubling property if there exists a constant $\delta<1$ such that 
\begin{equation}\label{reverse-dou}
	\frac{|B_I|_\omega}{|Q_I|_\omega}<\delta
\end{equation}
for all intervals $I\subset\T$. 
By \cite[Lemma 6]{FW2015}, a doubling weight on $\mathbb{D}$ has the reverse doubling property.
	 
\vspace{0.1in}

Now, we are ready to prove the sufficiency of Theorem \ref{main}.

Let $1<p<\infty,\frac{1}{p}+\frac{1}{p'}=1$ and $\omega$ be a weight on $\D$, and $\sigma=\omega^{1-p'}$
Assume both $\omega$ and $\sigma$ have the reverse doubling property and $$[\omega]_{p,\alpha}<\infty.$$
By Lemma \ref{lemma2.4}, it suffices to prove $T_\alpha^\beta$ is bounded on $L^p(\omega)$ for each $\beta\in\{0,\frac{1}{3}\}$. By Sawyer's dual arguments \cite{Sawyer1982}, this is equivalent to proving
\[
T^\beta_{\alpha,\sigma}:L^p(\sigma)\longrightarrow L^p(\omega)
\]
is bounded where $\sigma=\omega^{1-p'}$ and
\[
(T^\beta_{\alpha,\sigma}f)(z):=\left(T^\beta_\alpha (\sigma f)\right)(z)=\sum\limits_{I\in\mathcal{Q}^\beta}\frac{1}{|Q_I|^{\frac{\alpha}{2}}}\int_{Q_I}f(\lambda)\sigma(\lambda) dA(\lambda)\chi_{Q_I}(z).
\]
Next, we assume that $f$ and $g$ are non-negative functions on $\mathbb{D}$. For all $g\in L^{p'}(\omega)$, 
\begin{equation*}
	\begin{aligned}
		&\langle T^\beta_{\alpha,\sigma}f, g \rangle_{L^2(\omega)}\\
		=&\sum\limits_{I\in\mathcal{Q}^\beta}\frac{1}{|Q_I|^{\frac{\alpha}{2}}}
		\left(\int_{Q_I}f(\lambda) \sigma(\lambda) dA(\lambda)\right)
		\left(\int_{Q_I}g(\lambda) \omega(\lambda) dA(\lambda)\right)\\
		=&\sum\limits_{I\in\mathcal{Q}^\beta}\frac{|Q_I|_\omega |Q_I|_\sigma}{|Q_I|^{\frac{\alpha}{2}}}
		\left(\frac{1}{|Q_I|_\sigma}\int_{Q_I}f(\lambda) \sigma(\lambda) dA(\lambda)\right)
		\left(\frac{1}{|Q_I|_\omega}\int_{Q_I}g(\lambda) \omega(\lambda) dA(\lambda)\right).
	\end{aligned}
\end{equation*}
Here 
\begin{equation*}
	\begin{aligned}
		\frac{|Q_I|_\omega |Q_I|_\sigma}{|Q_I|^{\frac{\alpha}{2}}}
		&=\frac{|Q_I|_\omega |Q_I|_\sigma}{|Q_I|^{\frac{\alpha}{2}}}
		\frac{|Q_I|^{\frac{\alpha}{2}}}{|Q_I|_\omega^\frac{1}{p} |Q_I|_\sigma^{\frac{p-1}{p}}}
		\left(\frac{|Q_I|_\omega |Q_I|_\sigma^{p-1}}{|Q_I|^{\frac{p \alpha}{2}}}\right)^{\frac{1}{p}}\\
		&\leqslant[\omega]^{\frac{1}{p}}_{p,\alpha}|Q_I|_\omega^\frac{1}{p'} |Q_I|_\sigma^{\frac{1}{p}}.
	\end{aligned}
\end{equation*}
Hence,
\begin{equation*}
	\begin{aligned}
		&\langle T^\beta_{\alpha,\sigma}f, g \rangle_{L^2(\omega)}\\
		\leqslant&[\omega]^{\frac{1}{p}}_{p,\alpha} \sum\limits_{I\in\mathcal{Q}^\beta}
		\left[|Q_I|_\sigma^{\frac{1}{p}}\left(\frac{1}{|Q_I|_\sigma}\int_{Q_I}f(\lambda) \sigma(\lambda) dA(\lambda)\right)\right]
		\left[|Q_I|_\omega^\frac{1}{p'}\left(\frac{1}{|Q_I|_\omega}\int_{Q_I}g(\lambda) \omega(\lambda) dA(\lambda)\right)\right]\\
		\leqslant&[\omega]^{\frac{1}{p}}_{p,\alpha} 
		\left[\sum\limits_{I\in\mathcal{Q}^\beta}|Q_I|_\sigma\left(\frac{1}{|Q_I|_\sigma}\int_{Q_I}f(\lambda) \sigma(\lambda) dA(\lambda)\right)^p\right]^{\frac{1}{p}}
		\left[\sum\limits_{I\in\mathcal{Q}^\beta}|Q_I|_\omega\left(\frac{1}{|Q_I|_\omega}\int_{Q_I}g(\lambda) \omega(\lambda) dA(\lambda)\right)^{p'}\right]^\frac{1}{p'}\\
	\end{aligned}
\end{equation*}
Since $\sigma$ and $\omega$ have the reverse doubling property, by Lemma \ref{theorem2.5}, we have
\[
\langle T^\beta_{\alpha,\sigma}f, g \rangle_{L^2(\omega)}\leqslant C_6 [\omega]^{\frac{1}{p}}_{p,\alpha} ||f||_{L^p(\sigma)}||g||_{L^p(\omega)}.
\]
This completes the proof of the sufficient part, and completes the whole proof of Theorem \ref{main}.

\end{document}